\documentclass[reqno,centertags,a4paper,12pt]{amsart}
\usepackage{amssymb}
\usepackage{amsthm}
\usepackage{eucal}
\usepackage{color, soul, xcolor}
\soulregister\cite{7}
\soulregister\ref{7}
\soulregister\pageref7
\usepackage[english]{babel}
\usepackage{graphicx}
\usepackage{subfig}
\usepackage{hyperref}

\newcommand{\R}{\mathbb R}

\newcommand{\Z}{\mathbb Z}
\newcommand{\T}{\mathbb T}

\numberwithin{equation}{section}
\newtheorem{theorem}{Theorem}[section]
\newtheorem{proposition}[theorem]{Proposition}
\newtheorem{remark}[theorem]{Remark}
\newtheorem{lemma}[theorem]{Lemma}

\begin{document}
\title[Nonlinear Schr\"odinger equations with  third order dispersion]{Sharp global well-posedness for the cubic nonlinear Schr\"odinger equation with third order dispersion}

\author{X. Carvajal}
\address{Instituto de Matem\'atica, UFRJ, 21941-909, Rio de Janeiro, RJ, Brazil}
\email{carvajal@im.ufrj.br}
\author{M. Panthee}
\address{Department of Mathematics, University of Campinas\\
13083-859, Campinas, S\~ao Paulo, SP,  Brazil}
\email{mpanthee@unicamp.br}



\maketitle

\begin{abstract} We consider the
 initial value problem (IVP) associated  to  the  cubic nonlinear Schr\"odinger  equation with third-order dispersion
 \begin{equation*}
\partial_{t}u+i\alpha 
\partial^{2}_{x}u- \partial^{3}_{x}u+i\beta|u|^{2}u  =  0,
 \quad x,t \in \R, 
\end{equation*}
for given data in the Sobolev space $H^s(\R)$. This IVP is known to be locally well-posed for given data with Sobolev regularity $s>-\frac14$ and globally well-posed for $s\geq 0$  \cite{XC-04}.  For given data in $H^s(\R)$, $0>s> -\frac14$  no global well-posedness result is known. In this work, we derive an {\em almost conserved quantity} for such data and obtain a sharp global well-posedness result. Our result answers the question left open in \cite{XC-04}.

\end{abstract}


\section{Introduction}

In this work we  consider the initial value problem (IVP) associated to the  cubic nonlinear Schr\"odinger  equation with third-order dispersion 
 \begin{equation}\label{e-nls} 
  \begin{cases}
\partial_{t}u+i\alpha 
\partial^{2}_{x}u- \partial^{3}_{x}u+i\beta|u|^{2}u =  0,
 \quad x,t \in \R, \\
 u(x,0) = u_0(x),
 \end{cases}
\end{equation} 
where $\alpha,\beta\in \R $ and $u  = u(x, t)$ is  complex valued function.
 
 The equation in \eqref{e-nls}, also known as the extended nonlinear Schr\"odinger (e-NLS) equation,  appears to describe several physical phenomena like the nonlinear pulse propagation in an optical fiber, nonlinear modulation of a capillary gravity wave on water, for more details we refer to \cite{Agr-07}, \cite{XC-04}, \cite{DT-21}, \cite{HK-81}, \cite{MT-18}, \cite{Oikawa-93},  \cite{Tsutsumi-18} and references therein. In some literature, this model is also known as the third order Lugiato-Lefever equation \cite{MT-17} and can also be considered as a particular case of the higher order nonlinear Schr\"odinger (h-NLS)  equation proposed by  Hasegawa and Kodama in \cite{[H-K]} and \cite{[Ko]} to describe
 the  nonlinear propagation of pulses in optical fibers
\begin{equation*}\label{honse}
\begin{cases} 
\partial_{t}u-i\alpha 
\partial^{2}_{x}u+ \partial^{3}_{x}u-i\beta|u|^{2}u+\gamma
|u|^{2}\partial_{x}u+\delta\partial_{x}(|u|^2)u  =  0,
 \quad x,t \in \R,\\
  u(x,0) = u_0(x),
 \end{cases}
\end{equation*}
where $\alpha,\beta, \gamma\in \R $, $ \delta
  \in \mathbb{C}$ and  $u = u(x, t)$ are  complex valued function.

The well-posedness issues and  other properties of solutions of the IVP \eqref{e-nls}  posed on $\R$ or $\T$ have extensively been studied by several authors, see for example \cite{XC-04},  \cite{DT-21}, \cite{MT-17}, \cite{OTT-19}, \cite{CP-22} and references threrein. As far  as we know, the best local well-posedness result for the IVP~\eqref{e-nls} with given data in the $L^2$-based Sobolev spaces $H^s(\R)$, $s>-\frac14$,  is obtained by the first  author in \cite{XC-04}. More precisely, the following result was obtained in \cite{XC-04}. 
 
 \begin{theorem} \cite{XC-04}\label{LocalTh}
 Let $u_0\in H^s(\R)$ and $s>-\frac14$. Then  there exist
$\delta = \delta(\|u_0\|_{H^s})$ (with $\delta(\rho)\to \infty$ as $\rho\to 0$) and a unique solution to the IVP \eqref{e-nls} in the time interval $[0, \delta]$. Moreover, the solution satisfies the estimate
\begin{equation}\label{loc-est-2}
\|u\|_{X_{\delta}^{s, b}}\lesssim \|u_0\|_{H^s},
\end{equation}
 where the norm $\|u\|_{X_{\delta}^{s, b}}$ is as defined in \eqref{xsb-rest}.
 \end{theorem}
 
To obtain this result, the author in \cite{XC-04} derived a trilinear estimate
 \begin{equation}
 \label{tri-xc}
 \|u_1u_2\bar{u_3}\|_{X^{s,b'}}\lesssim \prod_{j=1}^3\|u_j\|_{X^{s,b}}^3, \quad 0\geq s>-\frac14,\;\;
b>\frac{7}{12}, \;\: b'<\frac{s}3, 
 \end{equation}
 where, for $s,b\in\R$, $X^{s,b}$ is the Fourier transform restriction norm space introduced by Bourgain \cite{B-93}  with norm
 \begin{equation}
 \label{X-norm}
 \|u\|_{X^{s,b}}:=\|\langle\xi\rangle^s\langle\tau-\phi(\xi)\rangle^b\widehat{u}(\xi, \tau)\|_{L^2_{\xi}L^2_{\tau}},
 \end{equation}
 where $\langle x\rangle:=1+|x|$ and $\phi(\xi) $ is the phase function associated to the e-NLS equation \eqref{e-nls} (for detailed definition, see \eqref{xsb} below). The author in \cite{XC-04} also showed that the crucial trilinear estimate \eqref{tri-xc} fails for $s<-\frac14$. Further, it has been proved that the application data to solution fails to be $C^3$ at the origin if $s<-\frac14$, see  Theorem 1.3, iv) in \cite{XC-13}.  In this sense, the local well-posedness result given by Theorem \ref{LocalTh} is sharp using this method.
 
 \begin{remark}
 We note that, the following quantity 
 \begin{equation}
 \label{conserved-1}
 E(u): = \int_{\R}|u(x,t)|^2 dx,
 \end{equation}
 is conserved by the flow of \eqref{e-nls}. Using this conserved quantity, the local solution given by Theorem \ref{LocalTh} can be extended globally in time, thereby proving the global well-posedness of the IVP \eqref{e-nls} in $H^s(\R)$, whenever $s\geq 0$.
 \end{remark}
 
 Looking at the local well-posedness result given by Theorem \ref{LocalTh} and the Remark above, it is clear that there is a gap between the local and  the global well-posedness results. In other words, one may ask the following natural question. Is it possible  that the local solution given by  Theorem \ref{LocalTh} can be extended globally in time for $ 0>s>-\frac14$? 
 
 The main objective of this work is to answer the question raised in the previous paragraph that is left open in \cite{XC-04} since 2004. In other words, the main focus of this work is  in investigating the global  well-posedness issue of the IVP \eqref{e-nls} for given data in the low regularity Sobolev spaces  $H^s(\R)$, $0>s>-\frac14$. No conserved quantities are available for data with regularity below $L^2$ to apply the classical method to extend the local solution globally in time. To overcome this difficulty  we use the  famous {\em I-method} introduced by Colliander et al \cite{CKSTT1, CKSTT2, CKSTT} and derive an {\em almost conserved quantity} to obtain the global 
 well-posedness result for given data in the low regularity Sobolev spaces. More precisely, the main result of this work is the following.
 
\begin{theorem}
\label{Global-Th}
The  IVP \eqref{e-nls} is globally well-posed for any initial data  $ u_0\in H^s(\R)$, $s>-\frac14$.
\end{theorem}

\begin{remark}
In the proof of this theorem, an almost conservation of the second generation of the  modified energy, viz.,
$$|E^2_I(u(\delta))|\leq |E^2_I(\phi)| + C N^{-\frac74}\|Iu\|_{X^{0, \frac12+}_{\delta}}^6$$
plays a crucial role. The decay  $N^{-\frac74}$ is more than enough to get the required result. Behind the proof of an almost conservation law, there are decay estimates of the multipliers involved. Structure of the multipliers in our case is different from the ones that appear in the case of the KdV or the NLS  equations, see for example  \cite{XC-06}, \cite{CKSTT2} and \cite{CKSTT}.  This fact creates some extra difficulties as can be seen in the proof of Proposition \ref{prop3.3}.
\end{remark}
 
 The well-posedness issues of the  IVP \eqref{e-nls} posed on the periodic domain $\T:=\R/2\pi\Z$ are also considered by several authors in recent time. The authors in \cite{MT-17} studied the  IVP \eqref{e-nls} considering that $\frac{2\alpha}3\notin \Z$ with data $u_0\in L^2(\T)$  and  obtained the global existence of the solution.  They also obtained the global attractor in $L^2(\T)$. The local existence result obtained in \cite{MT-17} is further improved  in  \cite{MT-18} for given data in the Sobolev spaces $H^s(\T)$  with $s>-\frac16$ (see also \cite{Tsutsumi-18}) with the same consideration. 
 
 Taking in consideration the  results in \cite{MT-17} and  \cite{MT-18}, there is a gap between the local and  the global well-posedness results in the periodic case too. In other words, one has the following natural question. Is it possible to extend to local solution to the IVP \eqref{e-nls} posed on periodic domain $\T$ can be extended globally in time for given data in $H^s(\T)$, $0>s>-\frac16$? Although this is a very good question, deriving {\em almost conserved quantities} in the periodic setting is more demanding and we will not consider it here.

 In recent time, other properties of solutions of the IVP \eqref{e-nls} have also been studied in the literature. The authors in \cite{OTT-19} proved  that the mean-zero Gaussian measures on Sobolev spaces $H^s(\T)$ are quasi-invariant under the flow whenever $s >\frac34$. This result is further improved in \cite{DT-21}  on Sobolev spaces $H^s(\T)$ for $s>\frac12$. Quite recently, in  \cite{CP-22}, we considered the IVP \eqref{e-nls}  with given  data in the modulation spaces $M_s^{2,p}(\R) $ and obtained the local well-posedness result for   $s> -\frac14$ and $2\leq p<\infty$.

Now we present the organization of this work. In Section
 \ref{sec-2}, we define  function spaces and provide some preliminary results. In Section \ref{sec-3} we introduce multilinear estimates and an {\em almost conservation law} that is fundamental to prove the  main result of this work. In Section \ref{sec-4} we provide the proof of the  main result of this paper.  
 We finish this section recording some standard notations that will be used throughout this work.\\

 \noindent
{\textbf{Notations:}} We use $c$ to denote various  constants whose exact values are immaterial and may
 vary from one line to the next. We use $A\lesssim B$ to denote an estimate
of the form $A\leq cB$ and $A\sim B$ if $A\leq cB$ and $B\leq cA$. Also, we
use the notation $a+$ to denote $a+\epsilon$ for $0< \epsilon \ll 1$.


\section{Function spaces and preliminary results}\label{sec-2}

We start this section by introducing some function spaces that will be used throughout this work. For $f:\R\times [0, T] \to \R$ we define the mixed
 $L_x^pL_T^q$-norm by
\begin{equation*}
\|f\|_{L_x^pL_T^q} = \left(\int_{\R}\left(\int_0^T |f(x, t)|^q\,dt
\right)^{p/q}\,dx\right)^{1/p},
\end{equation*}
with usual modifications when $p = \infty$. We replace $T$ by $t$ if $[0, T]$ is the whole real line $\R$.

We use $\widehat{f}(\xi)$ to denote  the Fourier transform of $f(x)$ defined by
$$
\widehat{f}(\xi) = c \int_{\R}e^{-ix\xi}f(x)dx$$
and
$\widetilde{f}(\xi)$ to denote  the Fourier transform of $f(x,t)$ defined by
$$
\widetilde{f}(\xi, \tau) = c \int_{\R^2}e^{-i(x\xi+t\tau)}f(x,t)dxdt.$$

We use $H^s$  to denote the $L^2$-based Sobolev space of order $s$ with norm
$$\|f\|_{H^s(\R)} = \|\langle \xi\rangle^s \widehat{f}\|_{L^2_{\xi}},$$
where $\langle \xi\rangle = 1+|\xi|$.

In order to simplify the presentation we  consider the following gauge transform considered in \cite{Takaoka}
\begin{equation}\label{gauge}
u(x,t) := v(x-d_1t, -t)e^{i(d_2x+d_3t)}.
\end{equation}

Using this transformation the IVP \eqref{e-nls} turns out to be
\begin{equation}\label{e-nls1} 
\!\!\!\!\!\!\!\begin{cases}
\partial_{t}v+ \partial^{3}_{x}v-i(\alpha -3d_2)
\partial^{2}_{x}v+(d_1+2\alpha d_2-3d_2^2)\partial_x v -i(d_2^3-\alpha d_2^2 +d_3) v - i\beta|v|^2v  =  0,\\
v(x,0) = v_0(x) := u_0(x)e^{-id_2x}.
\end{cases}
\end{equation} 
   
 If one chooses $d_1=-\frac{\alpha^2}3$, $d_2= \frac{\alpha}3$ and $d_3=\frac{2\alpha^3}{27}$                the third, fourth and fifth terms in the first equation in \eqref{e-nls1} vanish. Also, we note that
 $$\|u_0\|_{H^s}\sim \|v_0\|_{H^s}.
 $$

 So from now on, we will consider the IVP \eqref{e-nls} with $\alpha = 0$, more precisely,
\begin{equation}\label{e-nlsT} 
  \begin{cases}
\partial_{t}u+ \partial^{3}_{x}u-i\beta|u|^{2}u  =  0,
 \quad x,t \in \R, \\
 u(x,0) = u_0(x).
 \end{cases}
\end{equation}  

This simplification allows us to work in  the  Fourier transform restriction norm  space restricted to the cubic $\tau-\xi^3$. In what follows we formally introduce the Fourier transform restriction norm  space,  commonly known as the Bourgain's space.

For $s, b \in \R$,  we define the Fourier transform restriction norm space $X^{s,b}(\R\times\R)$ with norm
\begin{equation}\label{xsb}
\|f\|_{ X^{s, b}} = \|(1+D_t)^b U(t)f\|_{L^{2}_{t}(H^{s}_{x})} = \|\langle\tau-\xi^3\rangle^b\langle \xi\rangle^s \widetilde{f}\|_{L^2_{\xi,\tau}},
\end{equation}
 where $U(t) = e^{-t\partial^{3}_{x}}$ is the unitary group.

If $b> \frac12$, the Sobolev lemma imply that, $ X^{s, b} \subset C(\R ; H^s_x(\R)).$ For any interval $I$, we define the localized spaces $X^{s,b}_I:= X^{s,b}(\R\times I)$ with norm
\begin{equation}\label{xsb-rest}
\|f\|_{ X^{s, b}(\R\times I)} = \inf\big\{\|g\|_{X^{s, b}};\; g |_{\R\times I} = f\big\}.
\end{equation}
Sometimes we use the definition $X^{s,b}_{\delta}:=\|f\|_{ X^{s, b}(\R\times [0, \delta])}$.\\

We define a cut-off function $\psi_1 \in C^{\infty}(\R;\; \R^+)$ which is even, such that $0\leq \psi_1\leq 1$ and
\begin{equation}\label{cut-1}
\psi_1(t) = \begin{cases} 1, \quad |t|\leq 1,\\
                          0, \quad |t|\geq 2.
            \end{cases}
\end{equation}
We also define $\psi_T(t) = \psi_1(t/T)$, for $0\leq T\leq 1$.

In the following lemma we list some estimates that are crucial in the proof of the local well-posedness result whose proof can be found in \cite{GTV}.
\begin{lemma}\label{lemma1}
For any $s, b \in \R$, we have
\begin{equation}\label{lin.1}
\|\psi_1U(t)\phi\|_{X^{s,b}}\leq C \|\phi\|_{H^s}.
\end{equation}
Further, if  $-\frac12<b'\leq 0\leq b<b'+1$ and $0\leq \delta\leq 1$, then
\begin{equation}\label{nlin.1}
\|\psi_{\delta}\int_0^tU(t-t')f(u(t'))dt'\|_{X^{s,b}}\lesssim \delta^{1-b+b'}\|f(u)\|_{X^{s, b'}}.
\end{equation}

\end{lemma}

As mentioned in the introduction, our main objective is to prove the global well-posedness result for the low regularity data. Using the  $L^2$ conservation law \eqref{conserved-1} we have the global well-posedness of the IVP \eqref{e-nlsT} for given data in 
$H^s(\R),$ $s\geq 0$. So, from now on we suppose $0>s>-\frac14$ throughout
this work. \\

Our aim is to derive an {\em almost conserved quantity} and use it to prove Theorem  \ref{Global-Th}. For this, we  use the {\em I-method} introduced in \cite{CKSTT} and define the Fourier
multiplier operator $I$ by,
\begin{equation}\label{I-1}
\widehat{Iu}(\xi) = m(\xi) \widehat{u}(\xi),
\end{equation}
 where $m(\xi)$ is a smooth, radially symmetric and nonincreasing function given by
\begin{equation}\label{m-1}
m(\xi) =
\begin{cases}
1, \quad \quad \quad\,\,\,\,\,\,\,\;\;|\xi|< N, \\ 
N^{-s}|\xi|^s, \quad \quad |\xi|\geq 2N,
\end{cases}
\end{equation} 
with $N\gg 1$ to be fixed later.

Note that, $I$ is the identity operator in low frequencies, $\{\xi :
|\xi|< N\}$, and simply an integral operator in high frequencies. In
general, it commutes with differential operators and satisfies the
following property.

\begin{lemma} \label{gwplem1}
Let $0>s>-\frac14$ and $N\geq 1$. Then the operator $I$ maps $H^s(\R)$ to
$L^2(\R)$ and
\begin{equation}
\label{gwlem12}
\|I f\|_{L^2(\R)} \lesssim N^{-s}\|f\|_{H^s(\R)}.
\end{equation}
\end{lemma}
 
Now record a variant of the local well-posedness result for initial data $u_0 \in H^s$, $0>s>-\frac14$  such that $Iu_0\in L^2$. More precisely we have the following result which will be very useful in the proof of the global well-posedness theorem.

\begin{theorem}\label{local-variant}
Let $0>s>-\frac14$, then for any $u_0$ such that $Iu_0\in L^2$, there exist
$\delta = \delta(\|Iu_0\|_{L^2})$ (with $\delta(\rho)\to \infty$ as $\rho\to 0$) and a unique solution to the IVP \eqref{e-nlsT} in the time interval $[0, \delta]$. Moreover, the solution satisfies the estimate
\begin{equation}\label{variant-2}
\|Iu\|_{X_{\delta}^{0, b}}\lesssim \|Iu_0\|_{L^2},
\end{equation}
and the local existence time $\delta$ can be chosen satisfying
\begin{equation}\label{delta-var}
\delta \lesssim \|Iu_0\|_{L^2}^{-\theta},
\end{equation}
where $\theta>0$ is some constant.
\end{theorem}
\begin{proof}
As the operator $I$ commutes with the differential operators, the linear estimates in Lemma \ref{lemma1} necessary in the contraction mapping principle hold true after applying $I$ to equation \eqref{e-nlsT}. Since the operator $I$ does not commute with the nonlinearity, the trilinear estimate is not straightforward. However, applying the interpolation lemma (Lemma 12.1 in \cite{CKSTT-5}) to \eqref{tri-xc} we obtain, under the same assumptions on the parameters $s$, $b$ and $b'$ that
 \begin{equation}\label{tlin-I}
 \|I(u^3)_x\|_{X^{0, b'}}\lesssim \|Iu\|_{X^{0,b}}^3,
 \end{equation}
 where the implicit constant does not depend on the parameter $N$ appearing in the definition of the operator $I$.

 Now, using the trilinear estimate \eqref{tlin-I} and the linear estimates  the proof of this theorem follows exactly as in the proof of Theorem \ref{LocalTh}. So, we omit the details.
\end{proof}

We finish this section recording some known results that will be useful in our work.
First we record the following double mean value theorem (DMVT).
\begin{lemma}[DMVT]
Let $f\in C^2(\R)$, and $\max\{|\eta|,|\lambda|\}\ll |\xi|$, then
$$
|f(\xi+\eta +\lambda)-f(\xi+\eta )-f(\xi +\lambda)+f(\xi)|\lesssim|f''(\theta)\,|\eta|\,|\lambda|,
$$
where $|\theta| \sim |\xi|$.
\end{lemma}

The following Strichartz's type estimates will also be useful.
\begin{lemma}\label{Lema36} For any $s_1 \geq -\frac14$, $s_2 \geq 0$ and $b>1/2$, we have
\begin{align} 
\|u\|_{L_x^5 L_t^{10}} &\lesssim \|u\|_{X^{s_2,b}}\label{eqx1},\\
\|u\|_{L_x^{20/3} L_t^5} &\lesssim  \|u\|_{X^{s_1,b}}\label{eqx2},\\
\|u\|_{L_x^\infty L_t^\infty} &\lesssim  \|u\|_{X^{s_2,b}}\label{eqx3},\\
\|u\|_{L_x^2 L_t^2} &\lesssim  \|u\|_{X^{0,0}}\label{eqx4},\\
\|u\|_{L_t^\infty L_x^2 } &\lesssim  \|u\|_{X^{0,0}}\label{eqx5}.
\end{align}
\end{lemma}
\begin{proof}
The estimates \eqref{eqx1} and  \eqref{eqx2} follow from
$$
\|U(t)u_0\|_{L_x^5 L_t^{10}}\lesssim \|u_0\|_{L^2} \quad \textrm{and} \quad \|D_x^{\frac14}U(t)u_0\|_{L_x^{20/3} L_t^{5}}\lesssim \|u_0\|_{L^2},
$$
whose proofs can be found in \cite{[KPV1]}.
The estimates \eqref{eqx3} and \eqref{eqx5} follow by immersion and inequality \eqref{eqx4} is obviuous.
\end{proof}
\begin{lemma}
Let $n\geq 2$ be an even integer, $f_1,\dots,f_n \in \mathbf{S}(\R)$, then
$$
\int_{\xi_1+\cdots +\xi_n=0}\widehat{f_1}(\xi_1)\widehat{\overline{f_2}}(\xi_2)\cdots\widehat{f_{n-1}}(\xi_{n-1})\widehat{\overline{f_{n}}}(\xi_{n})=\int_{\R}f_1(x)\overline{f_2}(x)\cdots f_{n-1}(x)\overline{f_{n}}(x).
$$
\end{lemma}

\section{Almost Conservation law}\label{sec-3}

\subsection{Modified energy}

Before introducing modified energy functional, we define $n$-multiplier and $n$-linear functional.

Let $n\geq 2$ be an even integer. An $n$-multiplier $M_n(\xi_1, \dots, \xi_n)$ is a function defined on the hyper-plane $\Gamma_n:= \{(\xi_1, \dots, \xi_n);\;\xi_1+\dots +\xi_n =0\}$ with Dirac delta $\delta(\xi_1+\cdots +\xi_n)$ as a measure.

If $M_n$ is an $n$-multiplier and $f_1, \dots, f_n$ are functions on $\R$, we define an $n$-linear functional, as
\begin{equation}\label{n-linear}
\Lambda_n(M_n;\; f_1, \dots, f_n):= \int_{\Gamma_n}M_n(\xi_1, \dots, \xi_n)\prod_{j=1}^{n}\widehat{f_j}(\xi_j).
\end{equation}
When $f$ is a complex function and  $\Lambda_n$ is applied to the $n$ copies of the same function $f$, we write $$\Lambda_n(M_n)\equiv \Lambda_n(M_n; f):=\Lambda_n(M_n;\; f,\bar{f},f,\bar{f},\dots,f,\bar{f}).$$

For $1\leq j\leq n$ and $k\geq 1$, we define the elongation ${\bf{X}}_j^k(M_n)$ of the multiplier $M_n$ to be the multiplier of order $n+k$ given by
\begin{equation}\label{elong}
{\bf{X}}_j^k(M_n)(\xi_1, \cdots , \xi_{n+k} := M_n(\xi_1,\cdots,\xi_{j-1}, \xi_j+\cdots+\xi_{j+k}, \xi_{j+k+1}, \cdots, \xi_{n+k}). 
\end{equation}

Using Plancherel identity, the energy $E(u)$ defined in \eqref{conserved-1} can be written in terms of the $n$-linear functional as
\begin{equation}\label{e.2}
E(u)= \Lambda_2(1).
\end{equation}

In what follows we record a lemma that relates the time-derivative of the $n$-linear functional defined for the solution $u$ of the e-NLS equation \eqref{e-nlsT}.
\begin{lemma}\label{derivative}
Let $u$ be a solution of the IVP \eqref{e-nlsT} and $M_n$ be a $n$-multiplier, then
\begin{equation}\label{der.1}
\frac{d}{dt}\Lambda_n(M_n;u) = i\Lambda_n(M_n\gamma_n; u)+i\Lambda_{n+2}\Big(\sum_{j=1}^n\gamma_j^{\beta}{\bf{X}}_j^2(M_n;u)\Big),
\end{equation}
	where $\gamma_n = \xi_1^3+\cdots +\xi_n^3$, $\gamma_j^{\beta}=(-1)^{j-1}\beta$ and ${\bf{X}}_j^2(M_n)$ as defined in \eqref{elong}.
\end{lemma}

Now we introduce the first modified energy
\begin{equation}\label{mod-1}
E^1_I(u):= E(Iu),
\end{equation}
where $I$ is the Fourier multiplier operator defined in \eqref{I-1} with $m$ given by \eqref{m-1}. Note that for $m\equiv 1$, $E^1_I(u)= \|u\|_{L_2}^2 = \|u_0\|_{L_2}^2$.

Using Plancherel identity, we can write the first modified energy in terms of the $n$-linear functional as
\begin{equation}\label{mod-12}
\begin{split}
E^1_I(u)&= \int m(\xi)\widehat{u}(\xi)m(\xi)\bar{\widehat{u}}(\xi)d\xi\\
&=\int_{\xi_1+\xi_2=0} m(\xi_1)m(\xi_2)\widehat{u}(\xi_1)\widehat{\bar{u}}(\xi_2)\\
&=\Lambda_2(M_2; u),
\end{split}
\end{equation}
where $M_2=m_1m_2$ with $m_j =m(\xi_j)$, $j=1, 2$.

We define the second generation of the modified energy as
\begin{equation}\label{sec-m1}
E^2_I(u):= E^1_I(u)+\Lambda_4(M_4;u),
\end{equation}
where the multiplier $M_4$ is to be chosen later.

Now, using the identity \eqref{der.1}, we get
\begin{equation}\label{sec-m2}
\begin{split}
\frac{d}{dt} E^2_I(u) &= i\Lambda_2\Big(M_2\gamma_2;u\Big)+i\Lambda_4\Big(\sum_{j=1}^2\gamma_j^{\beta}{\bf{X}}_j^2(M_2);u\Big)\\
&\qquad +i\Lambda_4\Big(M_4\gamma_4;u\Big) + i\Lambda_6\Big(\sum_{j=1}^4\gamma_j^{\beta}{\bf{X}}_j^2(M_4); u\Big).
\end{split}
\end{equation}

Note that $\Lambda_2\big(M_2\gamma_2;u\big)=0$.
If we choose, $M_4$ in such a way that 
$$M_4\gamma_4+\sum_{j=1}^2\gamma_j^{\beta}{\bf{X}}_j^2(M_2)=0,$$
i.e.,
\begin{equation}\label{m.4}
M_4(\xi_1, \xi_2, \xi_3, \xi_4) = -\frac{\sum_{j=1}^2\gamma_j^{\beta}{\bf{X}}_j^2(M_2)}{\gamma_4},
\end{equation}
then we get $\Lambda_4 =0$ as well.

So, for the choice of $M_4$ in \eqref{m.4}, we have
\begin{equation}\label{second-m3}
\frac{d}{dt} E^2_I(u) =\Lambda_6(M_6),
\end{equation}
where 
\begin{equation}\label{m.6} 
M_6= \sum_{j=1}^4\gamma_j^{\beta}{\bf{X}}_j^2(M_4),
\end{equation}
with $M_4$ given by \eqref{m.4}.

We recall that on $\Lambda_n$ ($n=4,6$), one has $\xi_1+\cdots+\xi_n =0$. Let us introduce the notations $\xi_i+\xi_j=\xi_{ij}$, $\xi_{ijk} = \xi_i+\xi_j+\xi_k$ and so on.

Using the fact that $m$ is an even function, we can symmetrize the multiplier $M_4$ given by \eqref{m.4}, to obtain
\begin{equation}\label{m.4s}
\delta_4\equiv\delta_4(\xi_1, \xi_2, \xi_3, \xi_4):=[M_4]_{\mathrm{sym}} = \frac{\beta(m_1^2-m_2^2+m_3^2-m_4^2)}{6\xi_{12}\xi_{13}\xi_{14}},
\end{equation}
where we have used the identity $\xi_1^3+\xi_2^3+\xi_3^3= 3\xi_{12}\xi_{13}\xi_{14}$ on the hyperplane $\xi_1+\xi_2+\xi_3=0$. 

Using the  multiplier $[M_4]_{\mathrm{sym}}$ given by \eqref{m.4s} in \eqref{m.6} we obtain   $[M_6]_{\mathrm{sym}}$ in the symmetric form as follows
\begin{equation}\label{m.6s} 
\begin{split}
&\delta_6\equiv \delta_6(\xi_1, \xi_2, \xi_3, \xi_4, \xi_5, \xi_6)  :=[M_6]_{\mathrm{sym}}\\
&= \frac{\beta}{36}\sum_{\substack{\{k,m,o\}=\{1,3,5\}\\\{l,n,p\}=\{2,4,6\} }}\left[ \delta_4(\xi_{klm}, \xi_n, \xi_o, \xi_p)- \delta_4(\xi_k,\xi_{lmn}, \xi_o, \xi_p)+\delta_4(\xi_k,\xi_l,\xi_{mno}, \xi_p)-\delta_4(\xi_k, \xi_l, \xi_m, \xi_{nop})\right].
\end{split}
\end{equation}

\begin{remark}
In the case $k=1$, $l=2$, $m=3$, $n=4$, $o=5$, $p=6$, one can obtain the following sum of the symmetric multiplier  $[M_6]_{\mathrm{sym}}$ in the extended form as
\begin{equation*}
\begin{split}
 &\frac{\beta^2}{36}\Big[-\frac{m^2(\xi_{123})-m^2(\xi_4)+ m^2(\xi_5)-m^2(\xi_6)}{\xi_{56}\xi_{46}\xi_{45}}+\frac{m^2(\xi_1)-m^2(\xi_{234})+m^2(\xi_5)-m^2(\xi_6)}{\xi_{56}\xi_{15}\xi_{16}}\\
&\qquad  -\frac{m^2(\xi_1)-m^2(\xi_2)+m^2(\xi_{345})-m^2(\xi_{6})}{\xi_{12}\xi_{26}\xi_{16}} +\frac{m^2(\xi_1)-m^2(\xi_2)+m^2(\xi_3)-m^2(\xi_{456})}{\xi_{12}\xi_{13}\xi_{23}}\Big].
\end{split}
\end{equation*}
But, for our purpose $\delta_6$ given by \eqref{m.6s} in terms of $\delta_4$ is enough to obtain the required estimates, see Proposition \ref{prop3.3} below.
\end{remark}

\subsection{Multilinear estimates}

In this subsection we will derive some multilinear estimates associated to the symmetric multipliers $\delta_4$ and $\delta_6$, use them to get some local estimates in the Bourgain's space that will be useful to obtain an {\em almost conserved quantity}.

From here onwards we will consider the notation $|\xi_i|=N_i$, $m(N_i)=m_i$. Given four number $N_1, N_2, N_3, N_4$ and $\mathcal{C}=\{N_1, N_2, N_3, N_4\}$, we will denote $N_s=\max\mathcal{C}$, $N_a=\max\mathcal{C}\setminus\{N_s\}$, $N_t=\max\mathcal{C}\setminus\{N_s, N_a\}$, $N_b=\min\mathcal{C}$. Thus
$$
N_s \geq N_a\geq N_t \geq N_b.
$$
\begin{proposition}\label{prop3.3}
Let $m$ be as defined in \eqref{m-1}
\\
1) If $|\xi_{1j}| \gtrsim N_s$ for all $j=2,3,4$ and $|N_b|\ll N_s$, then
\begin{equation}\label{delta4.1}
|\delta_4| \sim \dfrac{m^2 (N_b)}{N_s^3}.
\end{equation}
2) If $|\xi_{1j}| \gtrsim N_s$ for all $j=3,4$ and $|\xi_{12}|\ll N_s$, then
\begin{equation}\label{delta4.2}
|\delta_4| \lesssim \dfrac{m^2 (N_b)}{\max\{N_t,N\} \,N_s^2}.
\end{equation}
3) If $|\xi_{1j}| \ll N_s$ for $j=2,3$,  $a\geq0$, $b\geq0$, $a+b=1$,  then
\begin{equation}\label{delta4.3}
|\delta_4| \lesssim \dfrac{m^2 (N_s)}{N_s^2|\xi_{12} |^{a}|\xi_{13} |^{b}}.
\end{equation}
4) In the other cases, we have
\begin{equation}\label{delta4.4}
|\delta_4| \lesssim \dfrac{m^2 (N_s)}{N_s^3}.
\end{equation}
\end{proposition}
\begin{proof}
Let $f(\xi):=m^2(\xi)$ be an even function, nonincreasing on $|\xi|$. From definition of $m(\xi)$, we have $|f'(\xi) |\sim \frac{m^2(\xi)}{|\xi|}$ if $|\xi|>N$.
Without loss of generality we can assume  $N_s=|\xi_1|$ and  $N_a=|\xi_2|$. As $N_s=|\xi_2+\xi_3+\xi_4|$, we have $N_a \sim N_s$. Also by symmetry we can assume $|\xi_{12}|\leq |\xi_{14}|$. 

By the definition of $\delta_4$, if $N_s\leq N$ then $\delta_4=0$. Thus so from now on, throughout the proof, we will consider that $N_s>N$. Depending on the frequency regimes we divide the proof in two different cases, viz., $|\xi_{13}| \gtrsim N_s$ and $|\xi_{14}| \gtrsim N_s$; and  $|\xi_{14}| \ll N_s$ or $|\xi_{13}| \ll N_s$.\\

\noindent
{\bf Case A. {\underline{$|\xi_{13}| \gtrsim N_s$ and $|\xi_{14}| \gtrsim N_s$:}}}  We further divide this case in two sub-cases.\\

\noindent
 {\bf Sub-case A1.  {\underline{$|\xi_{12}| \ll N_s$:}}}
 Using the standard  Mean Value Theorem, we have
\begin{equation}\label{eqA1}
|m^2 (\xi_1)-m^2 (\xi_2)|= |f (\xi_1)-f (-\xi_2)|
= |f'(\xi_{\theta_1})|\,|\xi_{12}|
\end{equation}
where  $\xi_{\theta_1}= \xi_1-\theta_1\xi_{12}$ with $\theta_1\in (0,1)$.

Since $|\xi_{12}|\ll N_s$ we have $|\xi_{\theta_1}|\sim |\xi_1|\sim N_s$ and consequently $|f'(\xi_{\theta_1})|\sim \dfrac{m^2 (N_s)}{N_s} $. Using this in \eqref{eqA1}, we obtain
\begin{equation}\label{eqA2}
\dfrac{|m^2 (\xi_1)-m^2 (\xi_2)|}{|\xi_{12}| |\xi_{13}| |\xi_{14}|}\lesssim \dfrac{m^2 (N_s)}{N_s^3}.
\end{equation}

Now, we move to estimate $|m^2(\xi_3)-m^2(\xi_4)|$. First note that, if $N_t \leq N$, then we have $|m^2(\xi_3)-m^2(\xi_4)|=0$. Thus we will assume that $|\xi_3|=N_t > N$.  We divide in two cases.\\

\noindent
{\bf Case 1. {\underline{$|\xi_{34}|\ll N_t$:}}} Using the Mean Value Theorem, we get
 \begin{equation}\label{eqA3}
|m^2 (\xi_3)-m^2 (\xi_4)|= |f (\xi_3)-f (-\xi_4)|=|f'(\xi_{\theta_2})|\,|\xi_{34}|,
\end{equation}
where  $\xi_{\theta_2}= \xi_3-\theta_2\xi_{34}$ with $\theta_2\in (0,1)$.
Since $|\xi_{34}|\ll N_t$ we have $|\xi_{\theta_2}|\sim |\xi_3|\sim N_t$ and consequently $|f'(\xi_{\theta_2})|\sim \dfrac{m^2 (N_t)}{N_t} $. Using this in \eqref{eqA3}, we obtain
 \begin{equation}\label{eqA4}
\dfrac{|m^2 (\xi_3)-m^2 (\xi_4)|}{{|\xi_{12}| |\xi_{13}| |\xi_{14}|}}
\lesssim \dfrac{m^2 (N_t)}{N_tN_s^2}.
\end{equation}

\noindent
{\bf Case 2. {\underline{$|\xi_{34}|\gtrsim N_t$:}}} In this case,  using triangular inequality and the fact that the function $f(\xi)=m^2(\xi)$  is  nonincreasing on $|\xi|$ we obtain from the 
 definition of $\delta_4$ that
\begin{equation}\label{eqA5}
 \dfrac{|m^2(\xi_3)-m^2(\xi_4)|}{|\xi_{12}|\,|\xi_{13}|\,|\xi_{14}|\,}\lesssim\dfrac{m^2 (N_b)}{N_t\,N_s^2}.
\end{equation}

 Now, combining \eqref{eqA2}, \eqref{eqA4} and \eqref{eqA5}, we obtain from the definition of $\delta_4$ in \eqref{m.4s} that
\begin{equation*}
|\delta_4| \sim \dfrac{|f(\xi_1)-f(\xi_2)+f(\xi_3)-f(\xi_4)|}{|\xi_{12}|\,|\xi_{13}|\,|\xi_{14}|\,}\lesssim\dfrac{m^2 (N_b)}{\max\{N_t,N\}\,N_s^2}.
\end{equation*}

\noindent
 {\bf Sub-case A2. {\underline{$|\xi_{12}| \gtrsim N_s$:}}} Here also, we divide in two different sub-cases.\\
 
 \noindent
 {\bf Sub-case A21. \underline{$N_b \gtrsim N_s$:}} In this case we have $N_b\sim N_t\sim N_a\sim N_s$.  Without loss of generality we can assume $\xi_1 > 0$. Since $\xi_1+\cdots+\xi_4=0$, two largest frequencies must have opposite signs, i.e., $\xi_2<0$. If possible, suppose $\xi_2 \geq 0$.  Then we have $\xi_{1}+\xi_{2}=:M>N_s$ and  $\xi_{3}+\xi_{4}=-M<-N_s<0$. In this situation one has $\xi_3 \xi_4>0$, otherwise
$$
\xi_{3}^2+\xi_{4}^2=M^2-2\xi_3 \xi_4\geq M^2>\xi_{1}^2+\xi_{2}^2,
$$
which is a contradiction. As $\xi_3+\xi_4<0$, we conclude  that $\xi_{3}<0$ and $\xi_{4}<0$. Now, the frequency ordering $|\xi_2|\geq |\xi_3|$ implies
 \begin{equation*}
\xi_2=M-\xi_1\geq |\xi_3|=-\xi_3 = M+\xi_4,
 \end{equation*}
and consequently $\xi_{14} \leq 0$. On the other hand, $|\xi_1|\geq |\xi_4| \implies \xi_1\geq -\xi_4 \implies \xi_{14}\geq 0$. Therefore, we get $\xi_{14}=0$ contradicting the hypothesis  $|\xi_{14}| \gtrsim N_s$ of  this case.
 
Now, for $\xi_1>0$ and  $\xi_2 < 0$, we have
 \begin{equation}\label{eqA21}
|m^2 (\xi_1)-m^2 (\xi_2)|= |f (\xi_1)-f (-\xi_2)|=|f'(\xi_\theta)|\,|\xi_{12}|,
\end{equation}
where $ \xi_1 \geq \xi_\theta\geq -\xi_2$, so that $\xi_\theta \sim N_s$ and consequently $|f'(\xi_\theta)|\sim \dfrac{m^2(N_s)}{N_s}$. Using this in \eqref{eqA21}, we get
\begin{equation}\label{eqA22}
|m^2 (\xi_1)-m^2 (\xi_2)|\lesssim m^2 (N_s).
\end{equation}

 Similarly, one can also obtain
 \begin{equation}\label{eqA23}
|m^2 (\xi_3)-m^2 (\xi_4)|
\lesssim m^2 (N_s).
 \end{equation}

Thus, taking in consideration of  \eqref{eqA22} and \eqref{eqA23}, from definition of $\delta_4$, we get
$$|\delta_4| \lesssim\dfrac{m^2 (N_s)}{N_s^3}.$$

 \noindent
 {\bf Sub-case A22. \underline{$N_b \ll N_s$:}} Without loss of generality we can assume $|\xi_4|=N_b$. In this case $|\xi_3|=|\xi_{12}+\xi_4|\sim |\xi_{12}|\sim N_s\sim |\xi_1| \sim |\xi_2|$. It follows that
 \begin{equation*}
|m^2 (\xi_1)-m^2 (\xi_2)+m^2 (\xi_3)-m^2 (\xi_4)|\sim |m^2 (\xi_4)|= |m^2 (N_b)|.
 \end{equation*}
 
 Therefore in this case
 $$
 |\delta_4| \sim\dfrac{m^2 (N_b)}{N_s^3}.
 $$

\noindent
 {\bf Case B. \underline{$|\xi_{14}| \ll N_s$ or $|\xi_{13}| \ll N_s$:}} We divide in two sub-cases.\\
  \noindent
 {\bf Sub-case B1. \underline{$|\xi_{14}| \ll N_s$:}} We move to find estimates considering two different sub-cases\\
 {\bf Sub-case B11. \underline {$|\xi_{13}|\gtrsim N_s$:}}  In this case we necessarily have $|\xi_{12}| \ll N_s$. If  $|\xi_{12}| \gtrsim N_s$, using the consideration made in the beginning of the proof, we get
 $$
 |\xi_{14}|\gtrsim|\xi_{12}|\gtrsim N_s,
 $$
 but this contradicts the defining condition $|\xi_{14}| \ll N_s$ of {\bf Case B1}.
 
 Now, for $|\xi_{12}| \ll N_s$ 
 using the Double Mean Value Theorem with $\xi:=-\xi_1$, $\eta:=\xi_{12}$ and $\lambda:=\xi_{14}$, we have
\begin{equation*}
\begin{split}
|f (\xi+\lambda+\eta)-f (\xi+\eta)-f (\xi+\lambda)+f(\xi)|&\lesssim|f''(\xi_\theta)|\, |\xi_{12}|\, \xi_{14}|\\
&\lesssim \dfrac{m^2 (N_s)|\xi_{12}|\,|\xi_{14}|}{N_s^2}.
\end{split}
\end{equation*}
Hence,
$$
 |\delta_4| \lesssim\dfrac{m^2 (N_s)|\xi_{12}|\,|\xi_{14}|}{N_s^2}\dfrac{1}{|\xi_{13}|\,|\xi_{12}|\,|\xi_{14}|}\sim \dfrac{m^2 (N_s)}{N_s^3}.
 $$

 \noindent
 {\bf Sub-case B12. {\underline{$|\xi_{13}|\ll N_s$:}}}
 Without loss of generality we can assume $\xi_1\geq 0$. Recall that, in this Sub-case  $|\xi_{12}|\leq |\xi_{14}| \ll N_s$. As $N_s=\xi_1$, we have
 \begin{equation*}
\begin{cases}
|\xi_{12}|\ll N_s &\Longrightarrow \,\,\xi_2 <0\quad \textrm{and} \quad |\xi_2|\sim N_s,\\
|\xi_{13}|\ll N_s &\Longrightarrow \,\,\xi_3 <0\quad \textrm{and} \quad |\xi_3|\sim N_s,\\
 |\xi_{14}|\ll N_s &\Longrightarrow \,\,\xi_4 <0\quad \textrm{and} \quad |\xi_4|\sim N_s.
\end{cases}
\end{equation*}

Combining these informations, we get
$$
N_s\gg |\xi_{13}|=|\xi_{24}|=|\xi_2|+|\xi_4|\sim N_s,
$$
which is a contradiction. Consequently this case  is not possible.
\\

\noindent
{\bf Sub-case B2. \underline{$|\xi_{13}| \ll N_s$:}} Taking in consideration {\bf Sub-case B1}, we will assume that $|\xi_{14}| \gtrsim N_s$. In this case too, we will analyse considering two different sub-cases.
\\
\noindent
 {\bf Sub-case B21. {\underline{$|\xi_{12}|\ll N_s$:}}}
 In this case we have $|\xi_1| \sim |\xi_2|\sim |\xi_3| \sim N_s$.  Furthermore
 $
 |\xi_4|=|\xi_{12}+ \xi_3|\sim N_s
 $.
Hence
 $$
 |\xi_1|\sim |\xi_2|\sim|\xi_3|\sim|\xi_4|\sim N_s.
 $$
Observe that, $N_a=|\xi_2|\geq |\xi_3|$ implies 
 $|\xi_{12}|\leq |\xi_{13}|$.
 In fact, if $\xi_1 > 0$, then
 \begin{equation*}
\begin{cases}
|\xi_{12}|\ll N_s &\Longrightarrow \,\,\xi_2 <0,\\
|\xi_{13}|\ll N_s &\Longrightarrow \,\,\xi_3 <0,
\end{cases}
\end{equation*}
and it follows that $\xi_{13}\geq \xi_{12}\geq 0$.

If $\xi_1 < 0$, then
 \begin{equation*}
\begin{cases}
|\xi_{12}|\ll N_s &\Longrightarrow \,\,\xi_2 >0,\\
|\xi_{13}|\ll N_s &\Longrightarrow \,\,\xi_3 >0,
\end{cases}
\end{equation*}
and it follows that $0\geq \xi_{12}\geq \xi_{13}$. Hence $|\xi_{12}|\leq |\xi_{13}|$.

On the other hand using the Mean Value Theorem, we obtain
\begin{equation*}
\begin{split}
|f(\xi_{1})-f(\xi_{2})+f(\xi_{3})-f(\xi_{4})|  &=|f(\xi_{1})-f(-\xi_{2})+f(\xi_{3})-f(-\xi_{4})| \\
&=|\xi_{12} f'(-\xi_2+\theta_1 \xi_{12})+\xi_{34} f'(-\xi_4+\theta_2 \xi_{34})|\\
&=|\xi_{12} |  \,|f'(-\xi_2+\theta_1 \xi_{12})-f'(-\xi_4+\theta_2 \xi_{34})|\\
&\lesssim |\xi_{12} |\,|f'(N_s)|\\
&\lesssim |\xi_{12} |\,\dfrac{m^2(N_s)}{N_s},
\end{split}
\end{equation*}
where  $|\theta_j|\leq 1$, $j=1,2$. From this we deduce
\begin{equation*}
|\delta_4| \lesssim \dfrac{|\xi_{12} |\, m^2 (N_s)}{N_s}\cdot\dfrac{1}{|\xi_{12} |\, |\xi_{13}|\, |\xi_{14}|}\leq\dfrac{ m^2 (N_s)}{N_s^2|\xi_{12} |^{a}|\xi_{13} |^{b}}.
\end{equation*}
\\
 {\bf Sub-case B22. {\underline{$|\xi_{12}|\gtrsim N_s$:}}}
As $N_a=|\xi_2|\sim |\xi_1|=N_s\sim |\xi_3|$, one has  $N_s\sim |\xi_2|=|\xi_{13}+\xi_4| $. Thus  $|\xi_4| \sim N_s$ and $|\xi_j| \sim N_s$, $j=1,2,3,4$. Also 
\begin{equation}\label{sign1324}
|\xi_{24}|=|\xi_{13}|\ll N_s \Longrightarrow \,\, \xi_3\,\xi_1 <0\quad \textrm{and} \quad\xi_2\,\xi_4 <0.
\end{equation}
Let 
\begin{equation}\label{eq1234}
\epsilon:=\xi_{13}=-\xi_{24}.
\end{equation}
We consider the following cases.
\\

\noindent
{\bf Case 1. {\underline{$\epsilon >0$}:}}  In this case if $\xi_1<0$, then $\xi_3=\epsilon +|\xi_1|> |\xi_1|$ which is a contradiction because $|\xi_3|\leq|\xi_1|$. Similarly by \eqref{sign1324} and \eqref{eq1234} if $\xi_2>0$, then $|\xi_4|=\epsilon +|\xi_2|> |\xi_2|$ which is a contradiction. Therefore we can assume  $\xi_1>0$ and $\xi_2<0$ and by \eqref{sign1324} $\xi_3<0$ and $\xi_4>0$. One has that
$$
\xi_1\geq -\xi_2\geq -\xi_3\geq \xi_4>0,
$$
and using \eqref{eq1234}
\begin{equation}\label{eq1234-10}
\xi_1\geq \xi_4+\epsilon\geq \xi_1-\epsilon \geq \xi_4>0.
\end{equation}
Let $b:=\xi_4-N_s$, using \eqref{eq1234-10}, we have
$$
N_s\geq N_s+b+\epsilon\geq N_s-\epsilon \geq N_s+b>0,
$$
which implies that $-2\epsilon\leq b\leq -\epsilon$. Consequently by \eqref{eq1234}, $\xi_2=-N_s-b-\epsilon$ and therefore using the condition of this {\bf Sub-case B22}
$$
N_s \lesssim \xi_{12}=-b-\epsilon,
$$
which is a contradiction. So, this case is not possible.
\\
\noindent
{\bf Case 2. {\underline{$\epsilon <0$}:}}  Similarly as above, if $\xi_1>0$, then $|\xi_3|=|\epsilon| +\xi_1>|\xi_1|$ which is a contradiction. Similarly by \eqref{sign1324} if $\xi_2<0$, then $\xi_4=|\epsilon| +|\xi_2|> |\xi_2|$ which is a contradiction. Therefore we can assume  $\xi_1<0$ and $\xi_2>0$ and by \eqref{sign1324} $\xi_3>0$ and $\xi_4<0$. Using \eqref{eq1234} one has that
\begin{equation}\label{eq1234-1}
-\xi_1\geq |\epsilon|-\xi_4\geq N_s-|\epsilon| \geq -\xi_4>0.
\end{equation}
Let $b:=\xi_4+N_s$, using \eqref{eq1234-1}, we have
$$
N_s\geq N_s-b+|\epsilon|\geq N_s-|\epsilon| \geq N_s-b>0,
$$
which implies that $2|\epsilon|\geq b\geq |\epsilon|$. Consequently $\xi_2=N_s-b+|\epsilon|$ and
$$
N_s \lesssim \xi_{12}=|\epsilon|-b,
$$
which is a contradiction. Therefore, this case also does not exist.

Combining all cases we finish the proof of proposition.
\end{proof}

\begin{remark}
Let $0<\epsilon\ll N_s$. An example for the {\bf Sub-case A1}  is
$$
\xi_1=N_s, \quad\xi_2=-N_s+\epsilon, \quad\xi_3=-\dfrac{\epsilon}{2} , \quad\xi_4=-\dfrac{\epsilon}{2},
$$
other example is
$$
\xi_1=N_s, \quad\xi_2=-N_s+\epsilon, \quad\xi_3=\dfrac{N_s}2-\dfrac{\epsilon}{2} , \quad\xi_4=-\dfrac{N_s}2-\dfrac{\epsilon}{2}.
$$
An example for the {\bf Sub-case A21}  with $\xi_1 \geq 0$ and $\xi_2 \leq 0$ is
$$
\xi_1=N_s, \quad\xi_2=-\dfrac{N_s}2, \quad\xi_3=-\dfrac{N_s}4, \quad\xi_4=-\dfrac{N_s}4.
$$
An example for the {\bf Sub-case A22} is
$$
\xi_1=N_s, \quad\xi_2=-\dfrac{N_s}{2}-\epsilon, \quad\xi_3=-\dfrac{N_s}{2} , \quad\xi_4=\epsilon.
$$
An example for the {\bf Sub-case B21} is
$$
\xi_1=N_s, \quad\xi_2=-N_s+\dfrac{\epsilon}2, \quad\xi_3=-N_s+\dfrac{\epsilon}{2} , \quad\xi_4=N_s-\epsilon.
$$

\end{remark}

\begin{proposition}
Let $w\in \mathbf{S}(\R \times \R)$, $0>s>-\frac14$ and $b>\frac12$, then we have
\begin{equation}\label{lambda4}
\left| \Lambda_4(\delta_4; u(t) ) dt\right| \lesssim \dfrac{1}{N^{(\frac54-3s)}}\,\|Iu\|^4_{L^2},
\end{equation}
and
\begin{equation}\label{lambda6}
\left| \int_0^\delta \Lambda_6(\delta_6; u(t) ) dt\right| \lesssim N^{-\frac74}\|Iu\|^6_{X_\delta^{0,b}}.
\end{equation}
\end{proposition}
\begin{proof}
To prove \eqref{lambda4}, taking idea from \cite{CKSTT2, CKSTT}, first we perform a Littlewood-Paley decomposition of the four factors $u$ on $\delta_4$ so that $\xi_j$ are essentially constants $N_j$, $j=1,2,3,4$. To recover the sum at the end we borrow a factor $N_s^{-\epsilon}$ from the large denominator $N_s$ and often this will not be mentioned. Also, without loss of generality, we can suppose that the Fourier transforms involved in the multipliers are all positive.

Recall that for $N_s\leq N$ one has $m(\xi_j)=1$ for all $j=1,2,3,4$ and consequently the  multiplier $\delta_4$  vanish. Therefore,  we will consider $N_s\leq N$.\\
In view of the estimates obtained in Proposition \ref{prop3.3}, we divide the proof of \eqref{lambda4} in two different parts.\\
\noindent
{\bf First part:   Cases 1), 2) and 4) of Proposition \ref{prop3.3}}. We observe that $N_s^{\frac14}m_s \gtrsim N^{-s}$. In fact, if $N_s\in [N,2N]$, then $m_s \sim 1$ and $N_s^{\frac14}m_s \gtrsim N_s^{-s}\gtrsim N^{-s}$. If $N_s>2N$, then from the definition of $m$ and the fact that $s>-\frac14$, we arrive at $N_s^{\frac14}m_s =N_s^{\frac14}\dfrac{N^{-s}}{N_s^{-s}}=N_s^{\frac14+s}N^{-s}\gtrsim~N^{-s}$. Furthermore, we observe that $\frac{1}{\max\{N_t,\, N\}} \leq \frac{1}{N}$. Thus
\begin{equation}\label{x1lambda45}
\begin{split}
\left| \Lambda_4(\delta_4; u(t) ) \right|& =\left|\int_{\xi_1+\cdots +\xi_4=0}\delta_4(\xi_1, \dots, \xi_4)\widehat{u_1}(\xi_1)\cdots\widehat{\overline{u_4}}(\xi_4)\right|\\
&\lesssim \int_{\xi_1+\cdots +\xi_4=0}\dfrac{m^2(N_b)}{N\,N_s^2}\dfrac{\widehat{Iu_1}(\xi_1)\cdots\widehat{I\overline{u_4}}(\xi_4)}{m_1\cdots m_4}\\
&\lesssim \int_{\xi_1+\cdots +\xi_4=0}\dfrac{N_s}{N\,N_s^2m_s^3}\widehat{D_x^{-\frac14}Iu_1}(\xi_1)\cdots\widehat{D_x^{-\frac14}I\overline{u_4}}(\xi_4)\\
&\lesssim \int_{\xi_1+\cdots +\xi_4=0}\dfrac{1}{N\,N_s^{\frac14-3s}}\widehat{D_x^{-\frac14}Iu_1}(\xi_1)\cdots\widehat{D_x^{-\frac14}I\overline{u_4}}(\xi_4)\\
&\lesssim \dfrac{1}{N\,N_s^{\frac14-3s}}\|D_x^{-1/4} Iu\|^4_{L^4}\\
&\lesssim \dfrac{1}{N^{(\frac54-3s)}}\|Iu\|^4_{L^2},
\end{split}
\end{equation}
where in the fourth line we used the following estimate
$$
\dfrac{N_s}{N_s^2 m_s^3}=\dfrac{1}{N_s^{\frac14}(N_s^{\frac14}m_s)^3}\lesssim\dfrac{1}{N_s^{\frac14-3s}}.
$$

\noindent
{\bf  Second part.  Case 3) of Proposition \ref{prop3.3}}. Recall from  the first part, we have  $N_s^{\frac14}m_s \gtrsim N^{-s}$. 
Using \eqref{delta4.3} with $a=1$ and $b=0$, and recalling the fact that $|\xi_{12}|= |\xi_{34}|$, we get  
\begin{equation}\label{x1lambda4}
\begin{split}
\left| \Lambda_4(\delta_4; u(t) ) \right|& =\left|\int_{\xi_1+\cdots +\xi_4=0}\delta_4(\xi_1, \dots, \xi_4)\widehat{u_1}(\xi_1)\cdots\widehat{\overline{u_4}}(\xi_4)\right|\\
&=\left|\int_{\xi_1+\cdots +\xi_4=0}\delta_4(\xi_1, \dots, \xi_4)\dfrac{\widehat{Iu_1}(\xi_1)\cdots\widehat{\overline{Iu_4}}(\xi_4)}{m_1\cdots m_4}\right|\\
&\lesssim\int_{\xi_1+\cdots +\xi_4=0}\dfrac{1}{N_s^2 m_s^2}|\xi_{12}|^{-\frac12}\widehat{Iu_1}(\xi_1)\widehat{\overline{Iu_2}}(\xi_2)\,|\xi_{34}|^{-\frac12}\widehat{Iu_3}(\xi_3)\widehat{\overline{Iu_4}}(\xi_4)
\\
&\lesssim \int_{\R}\dfrac{1}{N_s^{3/2-2s}}D_x^{-\frac12}(Iu_1Iu_2)\,D_x^{-\frac12}(Iu_3Iu_4)\\
&\lesssim \dfrac{1}{N_s^{\frac32-2s}}\|D_x^{-\frac12}(Iu_1Iu_2)\|_{L^2}\|D_x^{-\frac12}(Iu_3Iu_4)\|_{L^2},
\end{split}
\end{equation}
where in the second last line we used
$$
\dfrac{1}{N_s^2 m_s^2}=\dfrac{1}{N_s^{\frac32}(N_s^{\frac14}m_s)^2}\lesssim\dfrac{1}{N_s^{3/2-2s}}.
$$
Now, applying Hardy-Litlewwod-Sobolev inequality, we obtain from \eqref{x1lambda4} that
\begin{equation}\label{x1lambda4.1}
\begin{split}
\left| \Lambda_4(\delta_4; u(t) ) \right| &\lesssim \dfrac{1}{N^{(\frac32-2s)}}
\|Iu_1Iu_2\|_{L^1}\|Iu_3Iu_4\|_{L^1}\\
&\lesssim \dfrac{1}{N^{(\frac32-2s)}}\|Iu\|^4_{L^2}.
\end{split}
\end{equation}

Observe that the condition $s>-\frac14$ implies that $\frac54-3s<\frac32-2s$ and this completes the proof of \eqref{lambda4}.

Now we move to prove \eqref{lambda6}. As in the proof of \eqref{lambda4}, first  we perform a Littlewood-Paley decomposition of the six factors $u$ on $\delta_6$ so that $\xi_j$ are essentially constants $N_j$, $j=1, \cdots, 6$. Recall that for $N_s\leq N$ one has $m(\xi_j)=1$ for all $j=1,\cdots,6$ and consequently the  multiplier $\delta_6$  vanish. Therefore,  we will consider $N_s\leq N$. Since $N_a\sim N_s>N$, it follows that 
$$
m_sN_s\gtrsim N\quad \textrm{and}\quad m_aN_a\gtrsim N.
$$
Without loss of generality we will consider only the term $\delta_4(\xi_{123}, \xi_4, \xi_5, \xi_6)$ in the symmetrization of $\delta_6(\xi_1, \dots, \xi_6)$, see \eqref{m.6s}. The estimates for the other terms are similar.

  Here also, we will provide a proof of \eqref{lambda6} dividing in two parts.\\
  
\noindent
{\bf First part.  Cases 1), 2) and 4) in Proposition \ref{prop3.3}.} In these cases, we have 
\begin{equation}\label{x2lambda4}
\begin{split}
J&:=\left| \int_0^\delta \Lambda_6(\delta_6; u(t) ) \right| = \left|\int_0^\delta \int_{\xi_1+\cdots +\xi_6=0}\delta_4(\xi_{123}, \xi_4, \xi_5, \xi_6)\widehat{u_1}(\xi_1)\cdots\widehat{\overline{u_6}}(\xi_6)\right|\\
&\lesssim \int_0^\delta \int_{\xi_1+\cdots +\xi_6=0}\dfrac{m_b^2}{\max\{N_t, N\} \,N_s^2}\,\cdot\,\dfrac{m_a m_s\widehat{u_1}(\xi_1)\cdots\widehat{\overline{u_6}}(\xi_6)}{m_a m_s}\\
&\lesssim \int_0^\delta \int_{\R}\dfrac{1}{\max\{N_t, N\}\,N^2 }Iu_s Iu_a Iu_b u_t u_5 u_6\\\
&\lesssim \int_0^\delta \int_{\R}\dfrac{N_t^{\frac14}}{\max\{N_t, N\}\,N^2}Iu_s Iu_a Iu_b( D_x^{-\frac14}u_t) u_5 u_6\\
&\lesssim \dfrac{1}{N^{11/4}}\|Iu_s\|_{L_x^2 L_t^2} \|Iu_a\|_{L_x^\infty L_t^\infty}\|Iu_b\|_{L_x^\infty L_t^\infty}\|D_x^{-\frac14}u_t\|_{L_x^5 L_t^{10}}\|u_5\|_{L_x^{20/3} L_t^{5}}\|u_6\|_{L_x^{20/3} L_t^5}.
\end{split}
\end{equation}

Using estimates from Lemma \ref{Lema36}, we obtain from \eqref{x2lambda4} that
\begin{equation}\label{x2lambda4.1}
\begin{split}
J&\lesssim \dfrac{1}{N^{\frac{11}4}}\|Iu_s\|_{X^{0,b}_{\delta}} \|Iu_a\|_{X^{0,b}_{\delta}}\|Iu_b\|_{X^{0,b}_{\delta}}\|u_t\|_{X^{-\frac14,b}_{\delta}}\|u_5\|_{X^{-\frac14,b}_{\delta}}\|u_6\|_{X^{-\frac14,b}_{\delta}}\\
&\lesssim \dfrac{1}{N^{\frac{11}4}}\|Iu_s\|_{X^{0,b}_{\delta}} \|Iu_a\|_{X^{0,b}_{\delta}}\|Iu_b\|_{X^{0,b}_{\delta}}\|Iu_t\|_{X^{0,b}_{\delta}}\|Iu_5\|_{X^{0,b}_{\delta}}\|Iu_6\|_{X^{0,b}_{\delta}}\\
&\lesssim \dfrac{1}{N^{\frac{11}4}}\|Iu\|_{X^{0,b}_{\delta}}^6.
\end{split}
\end{equation}

\noindent
{\bf Second part. Case 3) in Proposition \ref{prop3.3}.} Without loss of generality we can assume that $|\xi_{123}|=N_s$, $|\xi_4|=N_a$, $|\xi_5|=N_t$ and $|\xi_6|=N_b$. Notice that $m_s^2\leq m_t m_b$ and $|\xi_j|\sim N_s$ for some $j=1,2,3$. So, we can assume $|\xi_3|\sim N_s$. Using \eqref{delta4.3} in Proposition \ref{prop3.3} with $a=1$, and $b=0$, we can obtain
\begin{equation}\label{x3lambda4}
\begin{split}
J:=&\left| \int_0^\delta \Lambda_6(\delta_6; u(t) ) \right| = \left|\int_0^\delta \int_{\xi_1+\cdots +\xi_6=0}\delta_4(\xi_{123}, \xi_4, \xi_5, \xi_6)\widehat{u_1}(\xi_1)\cdots\widehat{\overline{u_6}}(\xi_6)\right|\\
\lesssim &\int_0^\delta \int_{\xi_1+\cdots +\xi_6=0}\dfrac{m_t m_b}{N_s^2|\xi_{1234}|^{\frac12}\, |\xi_{56}|^{\frac12}}\,\cdot\,\dfrac{m_a  \widehat{u_1}(\xi_1)\cdots\widehat{\overline{u_6}}(\xi_6)}{m_a }\\
\lesssim &\int_0^\delta \int_{\xi_1+\cdots +\xi_6=0}\dfrac{N_s^{\frac14}}{N N_s}|\xi_{1234}|^{-\frac12}
(\widehat{u_1}(\xi_1)\widehat{\overline{u_2}}(\xi_2)|\xi_3|^{-\frac14}\widehat{u_3}(\xi_3))\widehat{\overline{Iu_4}}(\xi_4))
|\xi_{56}|^{-\frac12}(\widehat{Iu_5}(\xi_1)\widehat{\overline{Iu_6}}(\xi_6))\\
\lesssim &\dfrac{1}{N^{\frac74}}\int_0^\delta \int_{\R}D_x^{-\frac12}(u_1 u_2 (D_x^{-\frac14}u_3) Iu_a) D_x^{-\frac12}(Iu_t Iu_b)\\
\lesssim &\dfrac{1}{N^{\frac74}}\int_0^\delta \|  D_x^{-\frac12}(u_1 u_2 (D_x^{-\frac14}u_3) Iu_a)  \|_{L^2_x}\|D_x^{-\frac12}(Iu_t Iu_b)\|_{L^2_x}.
\end{split}
\end{equation}

Now, applying Hardy-Litlewwod-Sobolev inequality followed by estimates from Lemma \ref{Lema36}, we obtain from \eqref{x3lambda4} that
\begin{equation}\label{x3lambda4.1}
\begin{split}
J&\lesssim \dfrac{1}{N^{\frac74}}\int_0^\delta \|  u_1 u_2(D_x^{-\frac14}u_3) Iu_a \|_{L^1_x}\|Iu_t Iu_b\|_{L^1_x}\\
&\lesssim \dfrac{1}{N^{\frac74}} \|u_1\|_{ L_x^{20/3}L_t^5}\|u_2\|_{ L_x^{20/3}L_t^5}\|D_x^{-\frac14}u_3\|_{L_x^5 L_t^{10}}\|I u_a\|_{L_x^{2} L_t^{2}}\|Iu_t\|_{L_t^\infty L_x^2 }\|Iu_b\|_{L_t^\infty L_x^2 }\\
&\lesssim \dfrac{1}{N^{\frac74}}\|u_1\|_{X^{-\frac14,b}_{\delta}} \|u_2\|_{X^{-\frac14,b}_{\delta}}\|u_3\|_{X^{-\frac14,b}_{\delta}}\|Iu_a\|_{X^{0,b}_{\delta}}\|Iu_t\|_{X^{0,b}_{\delta}}\|Iu_b\|_{X^{0,b}_{\delta}}\\
&\lesssim \dfrac{1}{N^{\frac74}}\|Iu_1\|_{X^{0,b}_{\delta}} \|Iu_2\|_{X^{0,b}_{\delta}}\|Iu_3\|_{X^{0,b}_{\delta}}\|Iu_a\|_{X^{0,b}_{\delta}}\|Iu_t\|_{X^{0,b}_{\delta}}\|Iu_b\|_{X^{0,b}_{\delta}}\\
&\lesssim \dfrac{1}{N^{\frac74}}\|Iu\|_{X^{0,b}_{\delta}}^6.
\end{split}
\end{equation}
\end{proof}

\subsection{Almost conserved quantity}

We use the estimates proved in the previous subsection to obtain the following almost conservation law for the second generation of the energy.

\begin{proposition}\label{prop-almost}
Let $u$ be the solution of the IVP \eqref{e-nlsT} given by Theorem \ref{local-variant} in the interval $[0, \delta]$. Then the second generation of the modified energy satisfies the following estimates
\begin{equation}
\label{almost-CL}
|E^2_I(u(\delta))|\leq |E^2_I(\phi)| + C N^{-\frac74}\|Iu\|_{X^{0, \frac12+}_{\delta}}^6.
\end{equation}

\end{proposition}

\begin{proof}
The proof follows combining \eqref{second-m3} and \eqref{lambda6}.
\end{proof}

\section{proof of the main results}\label{sec-4}
In this section we provide proof of the  main results of this work. 

\begin{proof}[Proof of Theorem \ref{Global-Th}] Let $u_0\in H^s(\R)$, $s>0>-\frac14$. Given any $T>0$, we are interested in extending the local solution to the IVP \eqref{e-nlsT} to the interval $[0, T]$.

To make the analysis a bit easy we use the scaling argument. If $u(x,t)$ solves the IVP \eqref{e-nlsT} with initial data $u_0(x)$ then for $1<\lambda<\infty$, so does $u^{\lambda}(x,t)$ with initial data $u_0^{\lambda}(x)$; where  $u^{\lambda}(x,t)= \lambda^{-\frac32} u(\frac x\lambda, \frac t{\lambda^3})$ and $u_0^{\lambda}(x)=\lambda^{-\frac32}u_0(\frac x\lambda)$.

Our interest is in extending the rescaled solution $u^{\lambda}$ to the bigger time interval $[0, \lambda^3T]$.

Observe that
\begin{equation}\label{g-1}
\|u_0^{\lambda}\|_{{H}^s}\lesssim \lambda^{-1-s}\|u_0\|_{{H}^s}.
\end{equation}
From this observation and \eqref{gwlem12} we have that
\begin{equation}\label{g-2}
E^1_I(u_0^{\lambda})=\|Iu_0^{\lambda}\|_{L^2}^2\lesssim N^{-2s}\lambda^{-2(1+s)}\|u_0\|_{L^2}^2.
\end{equation}

The number $N\gg 1$ will be chosen later suitably. Now we choose the parameter $\lambda=\lambda(N)$ in such a way that $E^1_I(u_0^{\lambda})=\|Iu_0^{\lambda}\|_{L^2}^2$ becomes as small as we please. In fact, for arbitrary $\epsilon>0$, if we choose
\begin{equation}\label{g-4}
\lambda \sim N^{-\frac{s}{1+s}},
\end{equation}
 we can obtain
\begin{equation}\label{g-5}
E^1_I(u_0^{\lambda})=\|Iu_0^{\lambda}\|_{L^2}^2\leq \epsilon.
\end{equation}

From \eqref{g-5} and the variant of the local well-posedness result \eqref{delta-var}, we can guarantee that the rescaled solution $Iu^{\lambda}$ exists in the time interval $[0, 1]$.

Moreover, for this choice of $\lambda$, from \eqref{sec-m1}, \eqref{lambda4}   and \eqref{g-5},  in the time interval $[0, 1]$, we have
\begin{equation}\label{g-6}
|E^2_I(u_0^{\lambda})|\lesssim \|E^1_I(u_0^{\lambda})\|_{H^1}^2 +|\Lambda_4(M_4)|\lesssim \|Iu_0^{\lambda}\|_{L^2}^2 + \|Iu_0^{\lambda}\|_{L^2}^4\leq \epsilon+\epsilon^2\lesssim \epsilon.
\end{equation}

Using the almost conservation law \eqref{almost-CL} for the modified energy, \eqref{variant-2}, \eqref{g-5} and \eqref{g-6},  we obtain
\begin{equation}\label{g-7}
\begin{split}
|E^2_I(u^{\lambda})(1)|&\lesssim |E^2_I(u_0^{\lambda})| +N^{-\frac74}\|Iu^{\lambda}\|_{X_1^{0,{\frac12+}}}^6\\
&\lesssim \epsilon+N^{-\frac74}\epsilon^3\\
&\lesssim \epsilon+N^{-\frac74}\epsilon.
\end{split}
\end{equation}

From \eqref{g-7}, it is clear that we can iterate this process $N^{\frac74}$ times before doubling the modified energy $|E^2(u^{\lambda})|$. Therefore, by taking $N^{\frac74}$ times steps of size $O(1)$, we can extend the rescaled solution to the interval $[0, N^{\frac74}]$. As we are interested in extending the the solution to the interval $[0, \lambda^3T]$, we must select $N=N(T)$ such that $\lambda^3T\leq N^{\frac74}$. Therefore, with the choice of $\lambda$ in \eqref{g-4}, we must have
\begin{equation}\label{g-8}
TN^{\frac{-7-19s}{4(1+s)}}\leq c.
\end{equation}

Hence, for arbitrary $T>0$ and large $N$, \eqref{g-8} is possible if $s>-\frac7{19}$, which is true because we have considered $s>-\frac14$. This completes the proof of the theorem.
\end{proof}

\begin{remark}
From the proof of Theorem \ref{Global-Th} it can be seen that the global well-posedness result might hold for initial data with Sobolev regularity below $-\frac14$ as well provided there is local solution. But, as shown in \cite{XC-04} one cannot obtain the local well-posedness result for such data because the crucial trilinear estimate fails for for $s<-\frac14$.
\end{remark}


\vspace{0.7cm}
\noindent
{\bf Acknowledgments.} 
The first author extends thanks to  the Department of Mathematics, UNICAMP, Campinas for the kind hospitality where a significant  part of this work was developed. The second author acknowledges the grants from FAPESP (2020/14833-8) and CNPq (307790/2020-7). \\





\end{document}